\documentclass{amsart}

\usepackage{graphicx}%
\usepackage{multirow}%
\usepackage{amsmath,amssymb,amsfonts}%
\usepackage{amsthm}%
\usepackage{mathrsfs}%
\usepackage[title]{appendix}%
\usepackage{xcolor}%
\usepackage{textcomp}%
\usepackage{manyfoot}%
\usepackage{booktabs}%
\usepackage{algorithm}%
\usepackage{algorithmicx}%
\usepackage{algpseudocode}%
\usepackage{listings}%

\newtheorem{theorem}{Theorem}[section]
\newtheorem{corollary}[theorem]{Corollary}
\newtheorem{lemma}[theorem]{Lemma}
\newtheorem{proposition}[theorem]{Proposition}

\theoremstyle{definition}
\newtheorem{definition}[theorem]{Definition}
\theoremstyle{remark}
\newtheorem{remark}[theorem]{Remark}

\numberwithin{equation}{section}
\DeclareMathOperator{\card}{card}

\newcommand{\RR}{\mathbb{R}}

\begin{document}

\title[The existence of continuations for different types of metrics]{The existence of continuations for different types of metrics}

\author{Evgeniy Petrov}

\address{Institute of Applied Mathematics and Mechanics of the NAS of Ukraine, Batiuka str. 19, Slovyansk 84116, Ukraine}

\email{eugeniy.petrov@gmail.com}



\begin{abstract}The problems of continuation of a partially defined metric and a partially defined ultrametric were considered in~\cite{DMV13} and~\cite{DP13}, respectively.  Using the language of graph theory we generalize the criteria of existence of continuation obtained in these papers. For these purposes we use the concept of a triangle function introduced by M. Bessenyei and Z. P\'ales in~\cite{BP17}, which gives a generalization of the triangle inequality in metric spaces.
The obtained result allows us to get criteria of the existence of continuation for a wide class of semimetrics including not only metrics and ultrametrics, but also multiplicative metrics and semimetrics with power triangle inequality. Moreover, the explicit formula for the maximal continuations is also obtained.
\end{abstract}

\thanks{This work was partially supported by a grant from the Simons Foundation (Award 1160640, Presidential Dis\-cre\-ti\-ona\-ry-Ukraine Support Grants, E. Petrov).}



\maketitle

\section{Introduction}
Recall first some necessary definitions from the theory of metric spaces and graph theory. Let $X$ be a nonempty set.
A mapping  $d\colon X\times X\to \mathbb{R}^+$, $\mathbb{R}^+=[0,\infty)$ is a
\emph{metric} if for all $x,y,z \in X$ the following axioms hold:
\begin{itemize}
  \item [(i)] $(d(x,y)=0)\Leftrightarrow (x=y)$,
  \item [(ii)] $d(x,y)=d(y,x)$,
  \item [(iii)] $d(x,y)\leqslant d(x,z)+d(z,y)$.
\end{itemize}
The pair $(X,d)$ is called a \emph{metric space}.
If axiom (i) is replaced by a weaker condition
$d(x,x)=0$,
then $d$ is called a \emph{pseudometric}.
If axioms (i) and (ii) hold and condition (iii) is replaced by the inequality
\begin{itemize}
  \item [(iv)] $d(x,y)\leqslant \max\{d(x,z),d(z,y)\}$,
\end{itemize}
then $(X,d)$ is called an \emph{ultrametric space}. If additionally axiom (i) is replaced by the condition $d(x,x)=0$, then $d$ is called a \emph{pseudoultrametric}. If only axioms (i) and (ii) hold then $d$ is called a \emph{semimetric}. A pair $(X,d)$, where  $d$  is a semimetric on $X$, is called a \emph{semimetric space}.
In our paper we need a more general concept.
We shall say that $d$ is a \emph{pseudosemimetric} of only axiom (ii) and condition $d(x,x)=0$ hold. In this case the pair $(X,d)$ will be called a \emph{pseudosemimetric space}.

The fundamental concept of metric space was introduced by M. Fr\'{e}chet~\cite{Fr06} in 1906. Fr\'{e}chet called the discovered spaces ``classes (D)'' (from the word ``distance''). F. Hausdorff~\cite{Ha14} introduced the term ``metric space'' in 1914 considering these spaces as a special case of infinite topological spaces. Semimetric spaces were first examined by Fr\'{e}chet in~\cite{Fr06}, where he called them ``classes (E)''. Later these spaces attracted the attention of many mathematicians.
The ultrametric inequality was formulated by F.~Hausdorff in 1934 and ultrametric spaces were introduced by M. Krasner~\cite{Kr44} in 1944.

In 2017 M. Bessenyei and Z. P\'ales~\cite{BP17} introduced a definition of a triangle function $\Phi \colon \overline{\RR}_+^2\to \overline{\RR}^+$ for a semimetric $d$.
We use this definition in a slightly different form restricting the domain and the range of $\Phi$ by ${\RR}_+^2$ and ${\RR}^+$, respectively, and admitting $d$ to be a pseudosemimetric. Recall that a symmetry of $\Phi$ as usual means $\Phi(x,y)=\Phi(y,x)$.

\begin{definition}\label{d2}
Consider a pseudosemimetric space $(X, d)$. We shall say that $\Phi \colon {\RR}^+\times{\RR}^+ \to {\RR}^+$ is a \emph{triangle function} for $d$ if $\Phi$ is symmetric and monotone increasing in both of its arguments, satisfies $\Phi(0,0)=0$ and, for all $x, y, z \in X$, the generalized triangle inequality
\begin{equation}\label{ti}
d(x,y)\leqslant \Phi(d(x,z), d(y,z))
\end{equation}
holds. We also shall say that $d$ is a $\Phi$-pseudosemimetric if $\Phi$ is a triangle function for $d$.
\end{definition}

In~\cite{BP17} those semimetric spaces whose so-called basic triangle functions are continuous at the origin were considered. These spaces were termed regular. It was shown that the topology of a regular semimetric space is Hausdorff, that a convergent sequence in a regular semimetric space has a unique limit and possesses the Cauchy property, etc. See, e.g.,~\cite{CJT18,JT20,VH17} for some new results in this direction.

A graph $G$ is an ordered pair $(V,E)$ consisting of a set $V=V(G)$ of {\it vertices} and a set  $E=E(G)$ of {\it edges}. In this paper we study the {\it simple} graphs which are finite, $\operatorname{card}(V)<\infty$, or infinite, $\card(V)=\infty$. Since our graph $G$ is simple we can identify $E(G)$ with a set of two-element subsets of $V(G)$, so that each edge is an unordered pair of distinct vertices. As usual we suppose that $V(G)\cap E(G)=\varnothing$. The edge $e=\{u,v\}$ is said to {\it join} $u$ and $v$, and the vertices $u$ and $v$ are called {\it adjacent} in $G$. The graph $G$ is {\it empty} if no two vertices are adjacent, i.e. if $E(G)=\varnothing$. A graph $G$ is \emph{connected} if any two distinct vertices of $G$ can be joined by a path.
We use the standard definitions of the {\it path}, the {\it cycle}, the {\it subgraph} and {\it supergraph}, see, for example,~\cite[p.~4, p.~40]{BM}. Note only that all paths and cycles are finite and simple graphs.

The following, basic for us, notion is a {\it weighted graph} $(G,w)$, i.e., a graph $G=(V,E)$ together with a weight $w\colon E(G)\to\mathbb R^+$ where $\mathbb{R}^+=[0,\infty)$. The problem of continuation of a weight $w\colon E\to\mathbb R^+$ defined on the set of edges $E(G)$ of the graph $G=(V,E)$ to a pseudometric was considered in~\cite{DMV13}. In particular, it was found a set of necessary and sufficient conditions under which the weight $w$ can be extended to a pseudometric $d\colon V\times V\to\mathbb R^+$. The set of all such extensions can be partially ordered. It was shown that the shortest-path pseudometric $d_w$ is the greatest element of this set and that the least continuation exists if and only if $G$ is complete $k$-partite. The analogous problem of continuation of a weight to an ultrametric was considered in~\cite{DP13}. Moreover, in~\cite{DP13} the question of uniqueness of such continuation was studied. The uniqueness criterion for continuation of a weight to a pseudometric was found in~\cite{P22}. Note that a problem of continuation of a weight to a metric can be reformulated as a problem of continuation of a partially defined metric. In such setting, some cases of this problem were considered earlier. For example, the free amalgamation property for finite metric spaces states that there always exists a metric on the union $X\cup Y$ of finite metric spaces $(X,d_1)$, $(Y,d_2)$ which agrees with $d_1$ on $X$ and $d_2$ on $Y$, if $d_1=d_2$ for elements of $Z$ where $Z=X\cap Y$, see~\cite{B,AMI}. Further results connected with extensions of metrics and pseudometrics can be found in~\cite{H,TZ04,To72,Bi47,Be93,BSTZ17,BB00}.

Considering the criteria of continuation of a weight to pseudometric~\cite[Proposition 2.1]{DMV13} and to pseudoultrametric~\cite[Theorem 2]{DP13} and taking into consideration the concept of triangle function we naturally state a problem of generalization of these criteria. In other words we are interested in the following question. Let $(G,w)$ be a weighted graph. Does there exist a $\Phi$-pseudosemimetric $d\colon V(G)\times V(G)\to\RR^+$ such that the given weight $w\colon E(G)\to\mathbb R^+$ has a continuation to $d$? I.e., the equality
\begin{equation}\label{e2}
w(\{u,v\})=d(u,v)
\end{equation}
holds for all $\{u,v\}\in E(G)$. If such a continuation exists, then we say that $w$ is a {\it $\Phi$-pseudosemimetrizable weight}.

\section{Triangle functions}\label{TF}

Let $n\in \mathbb N$. For every triangle function $\Phi$ consider a function $\Phi^*\colon \mathbb R_{+}^{n}\to \mathbb R^{+}$ of $n$ variables, defined as
\begin{equation}\label{comp}
\Phi^*(x_1,...,x_n) =
 \begin{cases}
 x_1, &\text{if } n=1, \\
 \Phi(x_1,x_2), &\text{if } n=2, \\
 \Phi(x_1,\Phi(x_2,\Phi(x_3,... \Phi(x_{n-2},\Phi(x_{n-1},x_n))))), &\text{if } n\geqslant 3.
 \end{cases}
\end{equation}
It is clear that $\Phi^*$ is monotone increasing in all of its variables as well as $\Phi$.

Everywhere below we consider that the triangle function $\Phi$ satisfies the following three conditions.

(1) The equality
\begin{equation}\label{mc}
 \Phi(x,\Phi(y,z))=\Phi(z,\Phi(x,y)).
\end{equation}

(2)
The inequality
\begin{equation}\label{pr2}
  \Phi^*(a_1,...,a_n)\leqslant \Phi(\Phi^*(b_1,...,b_p), \Phi^*(c_1,...,c_q)),
\end{equation}
where $A=[a_1,...,a_n]$, $B=[b_1,...,b_p]$, $A=[c_1,...,c_q]$, $n, p, q\geqslant 1$, are multisets of nonnegative real numbers (i.e., the equality $a_i=a_j$, ($b_i=b_j$, $c_i=c_j$) is possible for $i\neq j$) and  $A\subseteq B\cup C$.

(3) The inequality
\begin{equation}\label{pr3}
  \Phi(a,b)\geqslant \max\{a,b\}.
\end{equation}

One can easily verify that the following functions satisfy the above mentioned conditions (1)--(3) and Definition~\ref{d2} of triangle functions:
\begin{itemize}
  \item [(i)] $\Phi(x,y)=x+y$,
  \item [(ii)] $\Phi(x,y)=\max \{x,y\}$,
  \item [(iii)] $\Phi(x,y)=x\cdot y$, in the case if $x,y\geqslant 1$,
  \item [(iv)] $\Phi(x,y)=(x^p+y^p)^{\frac{1}{p}}$, $p>0$,
  \item [(v)] $\Phi(x,y)=\varphi^{-1}(\varphi(x)+\varphi(y))$, where $\varphi\colon [0,\infty)\to [0,\infty)$ is a homeomorphism,
\end{itemize}
i.e., $\varphi(0)=0$, $\varphi$ is continuous, unbounded and strictly increasing.

Clearly, we obtain in (v) triangle function (iv) if $\varphi(x)=x^p$, $p>0$, and triangle function (i) if additionally $p=1$. It is interesting to note that a classical result of Acz\'{e}l~\cite[p.~107]{Ac87} ensures that~(\ref{mc}) together with additional conditions of continuity and cancellativity imposed on $\Phi$ guaranties that $\Phi(x,y)$ has form (v). The paper by Craigen and P\'{a}les~\cite{CP89} presents an alternative approach to this remarkable result. Note also that triangle function (ii) is not cancellative, i.e., the equality $\max\{x,y\}=\max\{z,y\}$ does not imply the equality $x=z$.

Obviously, metric spaces are semimetric spaces with triangle function (i) and ultrametric spaces are metric spaces with triangle function (ii). A power metric spaces were considered in~\cite{Gr16}. In these spaces  a power triangle function has a little more general form than (iv).

The notion of multiplicative metric space was introduced in~\cite{BKO08}. The main idea was that the usual triangle inequality was replaced by a multiplicative triangle inequality as follows.
Let $X$ be a nonempty set and $d\colon X\times X \to [1,\infty)$. We say that $(X,d)$ is a multiplicative metric space if for all $x,y,z\in X$ we have:
\begin{itemize}
  \item [$(i_1)$] $d(x,y)\geqslant 1$ and $x=y$ if and only if $d(x,y)=1$,
  \item [$(i_2)$] $d(x,y)=d(y,x)$,
  \item [$(i_3)$] $d(x,z)\leqslant d(x,y)d(y,z)$.
\end{itemize}
If condition $(i_1)$ is replaced by the weaker condition
\begin{itemize}
  \item [$(i_4)$] $d(x,x) = 1$,
\end{itemize}
then we say that $(X,d)$ is a pseudomultiplicative metric space.

\begin{remark}
Formally a pseudomultiplicative metric space is not a pseudosemimetric space with the triangle function $\Phi(x,y)=xy$. The reason is that condition $(i_4)$ contradicts to the axiom $d(x,x)=0$ of pseudosemimetric spaces. Nevertheless, analyzing the results of the paper we see that this fact does not violate the correctness of considerations related to continuation of a weight to a pseudomultiplicative metric.
\end{remark}

\begin{proposition}
Triangle function \emph{(v)} satisfies conditions \emph{(1)--(3)}.
\end{proposition}
\begin{proof}
(1)
$$ \Phi(x,\Phi(y,z)) =\varphi^{-1}(\varphi(x)+\varphi(\varphi^{-1}(\varphi(y)+\varphi(z))))
$$
$$
= \varphi^{-1}(\varphi(x)+\varphi(y)+\varphi(z))
=\Phi(z,\Phi(x,y)).
$$

(2) One can easily see that
$$
\Phi^*(a_1,...,a_n)= \varphi^{-1}(\varphi(a_1)+\cdots + \varphi(a_n))
$$
and
$$
\Phi(\Phi^*(b_1,...,b_p), \Phi^*(c_1,...,c_q))=
\varphi^{-1}(\varphi(b_1)+\cdots + \varphi(b_p)+\varphi(c_1)+\cdots + \varphi(c_q)).
$$
Using the inclusion $A\subseteq B\cup C$ and the fact that $\varphi$ and $\varphi^{-1}$ are strictly increasing we obtain inequality~(\ref{pr2}).

(3) Without loss of generality, consider that $\max\{a,b\}=a$. Using the evident equality $\Phi(a,0)=a$ and the fact that $\varphi$ and $\varphi^{-1}$ are strictly increasing we obtain inequality~(\ref{pr3}).
\end{proof}

Below we also need the following.
\begin{proposition}\label{p12}
Let $\Phi \colon {\RR}^+\!\times\!{\RR}^+ \!\to\!{\RR}^+$ be a symmetric function of two variables satisfying equality~(\ref{mc})
and let $x_1,...,x_n$, $n\geqslant 1$, be nonnegative real numbers.
Then
\begin{equation}\label{e111}
\Phi^*(x_1,...,x_n) = \Phi^*(\bar x_1,...,\bar x_n),
\end{equation}
where $(\bar x_1,...,\bar x_n)$ is any permutation of $(x_1,...,x_n)$.
\end{proposition}
\begin{proof}
It is well known that any permutation can be expressed as a composition  of transpositions. Moreover, any  transposition can be expressed as a composition of consecutive transpositions. Thus, to prove~(\ref{e111}) it suffices to show that
$$\Phi^*(x_1,...,x_i,x_{i+1},....,x_n)=\Phi^*(x_1,...,x_{i+1},x_{i},....,x_n)$$
for any $i=1,...,n-1$. Indeed, by symmetry of $\Phi$ and by~(\ref{mc}) we have
$$
 \Phi(x,\Phi(y,z))=\Phi(x,\Phi(z,y))=\Phi(y,\Phi(x,z)).
$$
Hence,
$$\Phi^*(x_1,...,x_i,x_{i+1},....,x_n)
= \Phi(x_1,...,\Phi(x_i,\Phi(x_{i+1},... \Phi(x_{n-2},\Phi(x_{n-1},x_n))))))$$
$$
= \Phi(x_1,...,\Phi(x_{i+1},\Phi(x_{i},... \Phi(x_{n-2},\Phi(x_{n-1},x_n))))))=
\Phi^*(x_1,...,x_{i+1},x_{i},....,x_n),
$$
which completes the proof.
\end{proof}

\section{Subdominant $\Phi$-pseudosemimetric}
In this section we introduce a concept of subdominant $\Phi$-pseudosemimetric, which is a generalization of such concepts as
\emph{shortest-path pseudometric}~\cite{DMV13} and  \emph{subdominant pseudoultrametric}~\cite{DP13}. Note that these two types of metrics play an important role in the problem of continuation of a weight to pseudometric and to pseudoultrametric, respectively.

Similarly to~\cite[p.~6]{DMV13} and to~\cite[p.~1134]{DP13} we give the following definitions. On the set of $\Phi$-pseudosemimetrics defined on $X$ we introduce the  partial order $\preceq$ as
\begin{equation}\label{eq2.02}
(d_1\preceq d_2) \  \text{ if and only if } \ (d_1(x,y)\leqslant d_2(x,y) \, \text{ for all  } \, x,y \in X).
\end{equation}

\begin{definition}\label{def2.3*}
Let $(G,w)$ be a nonempty weighted graph and $\mathfrak{F}_{\Phi}^w$ be the family of all $\Phi$-pseudosemimetrics $\rho\colon V(G)\times V(G)\to \RR^+$ such that
$$
\rho(u,v)\leqslant w(\{u,v\})
$$
for every edge $\{u,v\}\in E(G)$. If the poset $(\mathfrak{F}_{\Phi}^w,\preceq)$ contains the greatest element, then we call this element the \emph{subdominant $\Phi$-pseudosemimetric} for  $w$. Note that $\mathfrak{F}_{\Phi}^w \neq \varnothing$ because the zero pseudosemimetric $\rho(u,v)=0$, $u,v \in V(G)$, belongs to $\mathfrak{F}_{\Phi}^w$.
\end{definition}

Let $(G,w)$ be a weighted graph and let $u,v$ be vertices belonging to a connected component of $G$. Let us denote by $\mathcal P_{u,v}=\mathcal P_{u,v}(G)$ the set of all paths joining  $u$ and
$v$ in $G$. For the path $P\in \mathcal P_{u,v}$ define the \emph{$\Phi$-weight} of this path by
\begin{equation}\label{e31}
w_{\Phi}(P)=
\begin{cases}
0,  &\text{if } E(P)=\varnothing,\\
\Phi^*(w(e_1),...,w(e_n)), &\text{otherwise},
\end{cases}
\end{equation}
where $\Phi^*$ is defined by~(\ref{comp}) and $e_1,...,e_n$ are all edges of the path $P$.  By Proposition~\ref{p12} the function $w_{\Phi}$ is well-defined since it does not depend on the direction of the path. Write
\begin{equation}\label{e3}
d_{\Phi}^w(u,v)=\inf\{w_{\Phi}(P)\colon P\in\mathcal P_{u,v}\}.
\end{equation}
In the case $\Phi(x,y)=x+y$ for the connected graph $G$  the function $d^w_{\Phi}$ is a {shortest-path pseudometric} on the set $V(G)$ and in the case $\Phi(x,y)=\max\{x,y\}$ it is a subdominant pseudoultrametric.

In the next lemma and further we identify a pseudosemimetric space $(X,d)$ with the complete weighted graph $(G, w_d)=(G(X), w_d)$  having $V(G)=X$ and satisfying the equality
\begin{equation}\label{eq2.4}
	w_d(\{x,y\})=d(x,y)
\end{equation}
for every pair of different points $x,y \in X$.
\begin{lemma}\label{l1}
Let $(X,d)$ be a pseudosemimetric space with the triangle function $\Phi$. Then for every cycle $C\subseteq G(X)$ and for every $e\in E(C)$ the inequality
\begin{equation}\label{eq2.5}
	w_d(e) \leqslant w_{\Phi}(C\!\setminus\!e)
\end{equation}
holds, where $C\!\setminus \!e$ is a path obtained from the cycle $C$ by the removal of the edge $e$.
\end{lemma}
\begin{proof}
Let $V(C)=\{x_1,...,x_n\}$. Without loss of generality, consider that $e=\{x_1,x_2\}$. By~(\ref{ti}) we have
$$
w_d(e)=d(x_1,x_2)\leqslant \Phi(d(x_2,x_3),d(x_1,x_3)),
$$
$$
d(x_1,x_3)\leqslant \Phi(d(x_3,x_4),d(x_1,x_4)),
$$
$$
...
$$
$$
d(x_1,x_{n-1})\leqslant \Phi(d(x_{n-1},x_n),d(x_1,x_n)).
$$
Since $\Phi$ is strictly increasing in each of its arguments, we obtain
$$
w_d(e)\leqslant \Phi(d(x_2,x_3),\Phi(d(x_3,x_4),...,\Phi(d(x_{n-1},x_n),d(x_1,x_n))))
$$
$$
=\Phi^*(d(x_2,x_3),d(x_3,x_4),...,d(x_{n-1},x_n), d(x_1,x_n)).
$$
Clearly, $\{x_2,x_3\}$, ... ,$\{x_1,x_{n}\}$ are the edges of the path $C\!\setminus\!e$. Hence, by~(\ref{eq2.4}) and~(\ref{e31}) we obtain inequality~(\ref{eq2.5}).
\end{proof}

The proof of the following theorem is based on the proof of Theorem~1 from~\cite{DP13}. The main difference is that we use here the special properties of a triangle function defined by~(\ref{pr2}) and~(\ref{pr3}).

\begin{theorem}\label{st2.1}
{Let $\Phi$ be a continuous in both variables triangle function.} Then the function $d^{w}_{\Phi}$ is the subdominant $\Phi$-pseudosemimetric for every nonempty connected weighted graph $(G,w)$.
\end{theorem}

\begin{proof}
Let us verify the generalized triangle inequality
\begin{equation}\label{eq2.2}
d^{w}_{\Phi}(u,v)\leqslant \Phi(d^{w}_{\Phi}(u,p),d^{w}_{\Phi}(p,v))
\end{equation}
for different vertices $u,v,p \in V(G)$.
Let $\varepsilon$ be an  arbitrary positive number.
Then, by~(\ref{e3}) there exist paths $P_1 \in \mathcal{P}_{u,p}$ and $P_2 \in\mathcal{P}_{p,v}$ such that
\begin{equation}\label{eq2.3}
d^{w}_{\Phi}(u,p)+\varepsilon \geqslant w_{\Phi}(P_1)
\quad \mbox{and} \quad
d^{w}_{\Phi}(p,v)+\varepsilon \geqslant w_{\Phi}(P_2).
\end{equation}
The subgraph $G_1 $ of $G$ with $V(G_1)=V(P_1)\cup V(P_2)$ and $E(G_1)=E(P_1)\cup E(P_2)$ is connected.
Let $P_3$ be a path in $G_1$, connecting $u$ and $v$.
Using~(\ref{e3}), we have  $$d^{w}_{\Phi}(u,v) \leqslant w_{\Phi}(P_3).$$
Since $E(P_3)\subseteq E(P_1)\cup E(P_2)$ by~(\ref{e31}) and~(\ref{pr2}) we obtain
$$
w_{\Phi}(P_3) \leqslant \Phi(w_{\Phi}(P_1),w_{\Phi}(P_2)).
$$
Using the previous two inequalities,~(\ref{eq2.3}) and the monotonicity of $\Phi$, we get
$$
d^{w}_{\Phi}(u,v) \leqslant \Phi(d^{w}_{\Phi}(u,p)+\varepsilon,d^{w}_{\Phi}(p,v)+\varepsilon).
$$
Hence, letting $\varepsilon$ to zero, by the continuity of $\Phi$ we obtain~(\ref{eq2.2}).

It remains to verify that $d^{w}_{\Phi}$ is subdominant. Suppose there exist $\rho \in \mathfrak{F}_{\Phi}^w$ and $v_1, v_2 \in V(G)$ such that
\begin{equation}\label{eq2.9}
\rho(v_1,v_2)>d^w_{\Phi}(v_1,v_2).
\end{equation}
This inequality and~(\ref{e3}) imply the existence of a path  $P\in \mathcal{P}_{v_1,v_2}$ for which
\begin{equation}\label{eq2.10}
\rho(v_1,v_2)>w_{\Phi}(P).
\end{equation}
Note that $\rho(u,v)\leqslant w(\{u,v\})$ for every $\{u,v\}\in E(G)$. Consequently, in the case $\{v_1,v_2\}\in E(G)$ we have $w(v_1,v_2)>w_{\Phi}(P)$ and by inequality~(\ref{pr3}) the path $P$ does not contain $\{v_1,v_2\}$. Hence, in the pseudosemimetric space $(V(P),\rho)$ we can consider a cycle $C$ with
$$
V(C)=V(P), \quad E(C)=E(P)\cup\{\{v_1,v_2\}\}.
$$
As it was mentioned above we can associate with $(V(P),\rho)$ the weighted graph  $(G(V(P)),w_{\rho})$.
Since $w_{\rho}(\{u,v\})=\rho(u,v)\leqslant w(\{u,v\})$  by the monotonicity of functions $\Phi$, $\Phi^*$ and~(\ref{e31}) we have that $w_{\Phi}(P)\geqslant w_{\rho\Phi}(P)$. Hence, by~(\ref{eq2.10}) we have
$$
w_{\rho}(\{v_1,v_2\})=\rho(v_1,v_2)> w_{\rho\Phi}(P).
$$
Then the existence of such cycle in the pseudosemimetric space $(V(P),\rho)$ contradicts to Lemma~\ref{l1}. Thus, for every  $v_1,v_2 \in V(G)$ and  $\rho \in \mathfrak{F}_{\Phi}^w$ the inequality $\rho(v_1,v_2)\leqslant d_{\Phi}^w(v_1,v_2)$ holds, i.e., $d_{\Phi}^w$ is the greatest element of  $(\mathfrak{F}_{\Phi}^w,\preceq)$.
\end{proof}

\section{The existence of continuation}
If $(G,w)$ is a weighted graph with $\Phi$-pseudosemimetrizable $w$, then we shall denote by $\mathfrak{M}_{\Phi}^w$ the set of all $\Phi$-pseudosemimetrics $d\colon V(G)\times V(G)\to \mathbb{R}^+$ such that the equality~(\ref{e2}) holds  for all $\{u,v\}\in E(G)$. The proof of the following theorem is based on the proof of Proposition~2.1 from~\cite{DMV13} with the main difference that instead of the shortest-path pseudometric $d_w$ we use here the subdominant $\Phi$-pseudosemimetric $d_{\Phi}^w$.
\begin{theorem}\label{t1}
Let $(G,w)$ be a weighted graph and let $\Phi$ be a continuous {in both variables} triangle function. The following statements are equivalent.
\begin{itemize}
\item[$(i)$] The weight $w$ is $\Phi$-pseudosemimetrizable.
\item[$(ii)$] The equality
\begin{equation}\label{e4}
w(\{u,v\})=d_{\Phi}^w(u,v)
\end{equation}
holds for all $\{u,v\}\in E(G)$.
\item[$(iii)$] For every cycle $C\subseteq G$ and for every $e\in C$ the inequality
\begin{equation}\label{e5}
w(e)\leqslant w_{\Phi}(C\!\setminus \!e)
\end{equation}
holds, where $C\!\setminus \! e$ is a path obtained from $C$ by the removal of the edge $e$.
\end{itemize}
\end{theorem}

\begin{proof}
{$(i)\Rightarrow(ii)$} Suppose that there exists a $\Phi$-pseudosemimetric $\rho$ on $V(G)$ such that $w(\{u,v\})=\rho(u,v)$ for each $\{u,v\}\in E(G)$. Then by Lemma~\ref{l1} for every sequence of vertices $u=v_1,\dots ,v_n=v$, $v_i\in V(G)$, $i=1,\dots,n$, we have
$$
w(\{u,v\})=\rho(u,v)\leqslant w_{\Phi}(P),
$$
where $P=(v_1,...,v_n)$ is a path connecting $u$ and $v$ in the pseudosemimetric space $(V(G),\rho)$. Consequently, for all paths $P\subseteq G$ joining  $u$ and $v$ in $G$ the inequality $w(\{u,v\})\leqslant w_{\Phi}(P)$ holds. Passing in the last inequality to the infimum over the set $\{w_{\Phi}(P)\colon P\in\mathcal P_{u,v}\}$ we obtain
\begin{equation*}
\rho(u,v)=w(\{u,v\})\leqslant d_{\Phi}^w(u,v),
\end{equation*}
see \eqref{e3}. The converse inequality $w(\{u,v\})\geqslant d_{\Phi}^w(u,v)$ holds because the path $(u=v_1,v_2=v)$ belongs to $\mathcal P_{u,v}$.

{$(ii)\Rightarrow(iii)$} Suppose statement (ii) holds. Let $C$ be an arbitrary cycle in $G$ and let $e=\{u,v\}\in E(C)$ be an edge in $C$. Consider the path $P=C\!\setminus\!e$ joining  the vertices $u$ and $v$. Using equalities \eqref{e4} and \eqref{e3} we have
\begin{equation*}
w(e)=d_{\Phi}^w(u,v)\leqslant w_{\Phi}(P),
\end{equation*}
which completes the proof of this implication.

{$(iii)\Rightarrow(i)$} Suppose (iii) is true. If $G$ is a connected graph, then it is sufficient to show that
$d_{\Phi}^w\in\mathfrak{M}_{\Phi}^w$. Let $\{u,v\}\in E(G)$. In the case where there
is no cycle $C\subseteq G$ such that $\{u,v\}\in E(C)$ the path
$(u=v_1,v_2=v)$ is the unique path joining $u$ and $v$. Hence, in
this case, equality \eqref{e4} follows from \eqref{e3}. Let
$P=(u=v_1,\dots,v_{k+1}=v)$ be an arbitrary $k$-path, $k\geqslant 2$,
joining $u$ and $v$. Then $C=(u=v_1,\dots,v_{k+1},v_{k+2}=u)$ is a
$k+1$-cycle with $\{u,v\}\in E(C)$. Hence, by \eqref{e5} we have the inequality $w(\{u,v\})\leqslant w_{\Phi}(P) $ for all
$P\in\mathcal P_{u,v}$. Consequently, $ w(\{u,v\})\leqslant d_{\Phi}^w(u,v)$. The converse inequality is trivial. Thus, if $G$ is connected, then $d_{\Phi}^w\in\mathfrak{M}_w$.

Consider now the case of disconnected graph $G$. Let
$\{G_i\colon i\in\mathcal I\}$ be the set of all components of $G$ and let
$\{v_i^*\colon i\in\mathcal I\}$ be the subset of $V(G)$ such that
$
v_i^*\in V(G_i)
$
for each $i\in\mathcal I$. We choose an index $i_0\in\mathcal I$ and
fix nonnegative constants $a_i,\ i\in\mathcal I\setminus \{i_0\}$.
Consider the supergraph $G^*$ of $G$ such that
$V(G^*)=V(G)$ and
$$
E(G^*)=E(G)\cup\{\{v_i^*,v_{i_0}^*\}\colon i\in\mathcal
I\setminus\{i_0\}\}.
$$
All edges $\{v^*_i,v^*_{i_0}\}$ are bridges of
$G^*$. Now we can extend the weight $w\colon E(G)\to\mathbb R^+$ to a
weight $w^*\colon E(G^*)\to\mathbb R^+$ by the rule:
$$
w^*(\{u,v\})=\begin{cases} w(\{u,v\})&\text{if }\{u,v\}\in E(G),\\
a_i&\text{if }\{u,v\}=\{v_i^*,v_{i_0}^*\},\ i\in\mathcal
I\setminus\{i_0\}.
\end{cases}
$$
Since $(G^*,w^*)$  is a connected graph with the same set of
cycles as in $G$, it follows from implication {$(iii)\Rightarrow(i)$} that $d_{\Phi}^{w^*}\in\mathfrak{M}_{\Phi}^{w^*}$. Since the equality $w(\{u,v\})=d_{\Phi}^{w^*}(u,v)$ holds for all $\{u,v\} \in E(G)$ we have that $d_{\Phi}^{w^*}\in\mathfrak{M}_{\Phi}^{w}$.
\end{proof}

\begin{corollary}\label{c2}
{Let $\Phi$ be a continuous in both variables triangle function and} let $(G,w)$ be a weighted graph with a $\Phi$-pseudosemimetrizable weight $w$. If $G$ is a connected graph, then  $d_{\Phi}^{w}$ is the greatest element of the partially ordered set $(\mathfrak{M}_{\Phi}^{w}, \preceq)$. Conversely, if $(\mathfrak{M}_{\Phi}^{w}, \preceq)$ has the greatest element, then $G$ is a connected graph.
\end{corollary}
\begin{proof}
The first part of this corollary follows directly from Theorem~\ref{t1} and Theorem~\ref{st2.1}. The second part is trivial.
\end{proof}

Note that all functions (i)-(v) listed in Section~\ref{TF} are continuous in both variables. The equivalence of Statements (i) and (iii) in Theorem~\ref{t1} immediately implies the following corollaries.

\begin{corollary}\label{c3}
Let $(G,w)$ be a weighted graph. Then the weight $w$ is pseudoultrametrizable, i.e., $\Phi(x,y)=\max\{x,y\}$, if and only if for every cycle $C\subseteq G$ there exist at least two different edges $e_1, e_2 \in E(C)$ such that
\begin{equation}\label{in2}
w(e_1)=w(e_2)=\max\limits_{e\in E(C)}w(e).
\end{equation}
\end{corollary}
\begin{proof}
Setting in~(\ref{e5}) $\Phi(x,y)=\max\{x,y\}$, we get
$$
w(e)\leqslant \max\limits_{\tilde{e}\in C\setminus e}w(\tilde{e})
$$
for every $e\in C$.
Taking instead of $e$ the edge $e_m$ of the maximal weight in the cycle $C$ we get
$$
w(e_m)=\max\limits_{\tilde{e}\in C}w(\tilde e)\leqslant \max\limits_{\tilde{e}\in C\setminus e_m}w(\tilde{e})
$$
This means that the edge of maximal weight is represented at least twice in the cycle $C$ which is equivalent to~(\ref{in2}).
\end{proof}
Note that Corollary~\ref{c3} is exactly Theorem~2 from~\cite{DP13}.

\begin{corollary}\label{c4}
Let $(G,w)$ be a weighted graph. Then the weight $w$ is $\Phi$-pseudosemimetrizable with $\Phi(x,y)=\varphi^{-1}(\varphi(x)+\varphi(y))$, where $\varphi\colon [0,\infty)\to [0,\infty)$ is a homeomorphism, if and only if
for every cycle $C\subseteq G$ and every $e\in C$ the inequality
\begin{equation*}
 w(e)\leqslant \varphi^{-1}\left(\sum\limits_{\tilde e\in C\setminus e}\varphi(w(\tilde e))\right)
\end{equation*}
holds.
\end{corollary}
\begin{proof}
This equivalence can be obtained by the direct substitution of the corresponding function $\Phi$ into~(\ref{e5}).
\end{proof}

\begin{remark}
Setting in Corollary~\ref{c4}  $\varphi(t)=t^p$, $p>0$, we get that the weight $w$ is $\Phi$-pseudosemimetrizable with $\Phi(x,y)=(x^p+y^p)^{\frac{1}{p}}$ if and only if for every cycle $C\subseteq G$ and every $e\in C$ the inequality
\begin{equation}\label{ce51}
 w(e)\leqslant \left(\sum\limits_{\tilde e\in C\setminus e}w^p(\tilde e)\right)^{\frac{1}{p}}
\end{equation}
holds. Setting in~(\ref{ce51}) $p=1$ and taking instead of $e$ an edge of the maximal weight in the cycle $C$, we get that the weight $w$ is pseudometrizable, i.e., $\Phi(x,y)=x+y$, if and only if for every cycle $C\subseteq G$ the inequality
\begin{equation*}
2\max\limits_{e\in E(C)} w(e)\leqslant \sum\limits_{e\in C}w(e)
\end{equation*}
holds. Note that this equivalence is exactly Proposition~2.1 from~\cite{DMV13}.
\end{remark}

\begin{corollary}
The weight $w$ has a continuation to a pseudomultiplicative metric, i.e., $\Phi(x,y)=xy$, if and only if for every cycle $C\subseteq G$ and every $e\in C$ the inequality
\begin{equation*}
 w(e)\leqslant \prod\limits_{\tilde e\in C\setminus e}w(\tilde e)
\end{equation*}
holds and $w({e})\geqslant 1$ for every ${e}\in E(G)$.
\end{corollary}
\begin{proof}
This equivalence can be obtained by the direct substitution of the corresponding function $\Phi$ into~(\ref{e5}). Note also that the condition $w(e)\geqslant 1$ for every $e\in E(G)$ is a necessary condition for the continuation of the weight to a pseudomultiplicative metric. Moreover, the function $\Phi(x,y)=xy$ satisfies conditions~(\ref{pr2}) and~(\ref{pr3}) only if its arguments $x,y\geqslant 1$.
\end{proof}

\textbf{Acknowledgements.} This work was partially supported by a grant from the Simons Foundation (Award 1160640, Presidential Discretionary-Ukraine Support Grants, E. Petrov).

\end{document}